\documentclass[a4paper, 12pt, leqno]{amsart}
\usepackage{amsmath}
\usepackage{amscd}
\usepackage{amssymb}
\usepackage{amsthm}
\usepackage{amsfonts}
\usepackage{latexsym}

\usepackage[dvipdfmx]{graphicx, xcolor}
 
\theoremstyle{plain}
\newtheorem{lemma}{{\sc Lemma}}[section]

\newtheorem{proposition}[lemma]{{\sc Proposition}}
\newtheorem{theorem}[lemma]{{\sc Theorem}}

\theoremstyle{definition}
\newtheorem{remark}[lemma]{{\sc Remark}}

\numberwithin{equation}{section}

\def\Gg{{\mathfrak{g}}}
\def\Gh{{\mathfrak{h}}}

\def\BA{{\mathbb{A}}}

\def\BF{{\mathbb{F}}}

\def\BQ{{\mathbb{Q}}}
\def\BZ{{\mathbb{Z}}}

\def\CE{{\mathcal E}}

\def\CO{{\mathcal O}}

\def\ch{\mathop{\rm ch}\nolimits}

\def\Ht{{\mathop{\rm ht}\nolimits}}
\def\id{\mathop{\rm id}\nolimits}

\def\int{{\mathop{\rm int}\nolimits}}

\def\rank{{\mathop{\rm rank}\nolimits}}

\def\sgn{{\mathop{\rm sgn}\nolimits}}

\def\tA{{\tilde{\BA}}}
\def\oU{{\overline{U}}}

\begin{document}
\title[Characters of integrable highest weight modules]
{Characters of integrable highest weight modules over a quantum group}
\author{Toshiyuki TANISAKI}
\thanks
{
The author was partially supported by Grants-in-Aid for Scientific Research (C) 15K04790
from Japan Society for the Promotion of Science.
}
\address{
Department of Mathematics, Osaka City University, 3-3-138, Sugimoto, Sumiyoshi-ku, Osaka, 558-8585 Japan}
\email{tanisaki@sci.osaka-cu.ac.jp}
\date{}
\begin{abstract}
We show that the Weyl-Kac type character formula holds
for the integrable highest weight modules over 
the quantized enveloping algebra of any symmetrizable Kac-Moody Lie algebra, when 
the parameter $q$ is not a root of unity.
\end{abstract}
\maketitle

\section{Introduction}

It is well-known that the character of an integrable highest weight module over a symmetrizable Kac-Moody algebra $\Gg$ is given by the Weyl-Kac character formula (see Kac \cite{K}).
In this paper we  consider the corresponding problem  for a quantized enveloping algebra (see Kashiwara \cite{Kas}).

For a field $K$ and $z\in K^\times$ which is not a root of 1, we denote by
$U_{K,z}(\Gg)$ the quantized enveloping algebra of $\Gg$ over $K$ at $q=z$, namely the specialization of Lusztig's $\BZ[q,q^{-1}]$-form via $q\mapsto z$.
It is already known that the Weyl-Kac type character formula holds for $U_{K,z}(\Gg)$ in some cases.
When $K$ is of characteristic $0$ and $z$ is transcendental, this is due to Lusztig \cite{L0}.
When $\Gg$ is finite-dimensional,  this is shown in Andersen, Polo and Wen \cite{APW}.
When $\Gg$ is affine, 
this is known in certain specific cases  (see Chari and Jing \cite{CG}, Tsuchioka \cite{Tsu}).

We first point out that
the problem is closely related to the non-degeneracy of the Drinfeld pairing for $U_{K,z}(\Gg)$.
In fact, assume we could show that the Drinfeld pairing for $U_{K,z}(\Gg)$ is non-degenerate.
Then we can define the quantum Casimir operator.
It allows us to apply Kac's argument for Lie algebras in \cite{K} to $U_{K,z}(\Gg)$, and we obtain the Weyl-Kac type character formula for integrable highest weight modules over $U_{K,z}(\Gg)$.
In particular, we can deduce the Weyl-Kac type character formula in the affine case from 
the  case-by-case calculation of the Drinfeld pairing  due to  Damiani \cite{Da1}, \cite{Da2}.

The aim of this paper is to give a simple unified proof of the non-degeneracy of the Drinfeld pairing and the Weyl-Kac type character formula 
for  $U_{K,z}(\Gg)$, where $\Gg$ is a symmetrizable Kac-Moody algebra, $K$ is a field not necessarily of characteristic zero, and $z\in K^\times$ is not a root of 1.
Our argument is as follows.
We consider the (possibly) modified algebra $\oU_{K,z}(\Gg)$, which is the quotient of $U_{K,z}(\Gg)$ by the ideal generated by the radical of the Drinfeld pairing.
Then the Drinfeld pairing for $U_{K,z}(\Gg)$ induces a non-degenerate pairing for $\oU_{K,z}(\Gg)$, by which we can define the quantum Casimir operator for $\oU_{K,z}(\Gg)$.
It allows us to apply  Kac's argument for Lie algebras to $\oU_{K,z}(\Gg)$, and we obtain the Weyl-Kac type character  formula for $\oU_{K,z}(\Gg)$ with modified denominator.
In the special case where the highest weight is zero, this gives a formula for the modified denominator.
Comparing this with the ordinary denominator formula for Lie algebras, we conclude that the  modified denominator coincides with the original denominator for the Lie algebra $\Gg$.
It implies that  the Drinfeld pairing for $U_{K,z}(\Gg)$ was already non-degenerate.
This is the outline of our argument.
In applying Kac's argument to the modified algebra, we need to show that the modified denominator is skew invariant with respect to a twisted action of the Weyl group.
This is accomplished using certain standard properties of the Drinfeld pairing.

The first draft of this paper contained only  results when $K$ is of characteristic zero.
Then Masaki Kashiwara pointed out to me that the arguments work for positive characteristic case as well.
I would like to thank Masaki Kashiwara for this crucial remark.

\section{quantized enveloping algebras}
Let $\Gh$ be a finite-dimensional vector space over $\BQ$, and 
let
$\{h_i\}_{i\in I}$ and $\{\alpha_i\}_{i\in I}$   be
linearly independent subsets of $\Gh$ and $\Gh^*$, respectively such that 
$(\langle \alpha_j,h_i\rangle)_{i, j\in I}$ is a symmetrizable generalized Cartan matrix.
We denote by $W$ the associated Weyl group.
It is a subgroup of $GL(\Gh)$ generated by the involutions 
$s_i$ ($i\in I$) defined  by $s_i(h)=h-\langle\alpha_i,h\rangle h_i$ for $h\in \Gh$.
The contragredient action of $W$ on $\Gh^*$ is given by  $s_i(\lambda)=\lambda-\langle\lambda,h_i\rangle \alpha_i$ for $i\in I$, $\lambda\in \Gh^*$.
Set 
\[
E=\sum_{i\in I}\BQ\alpha_i,
\qquad
Q=\sum_{i\in I}\BZ\alpha_{i},\qquad
Q^+=\sum_{i\in I}\BZ_{\geqq0}\alpha_i.
\]
We can take a symmetric $W$-invariant bilinear form 
$
(\;,\;):
E\times E\to\BQ
$
such that
\begin{equation}
\frac{(\alpha_i,\alpha_i)}2\in\BZ_{>0}\qquad(i\in I).
\end{equation}
For $\lambda\in E$ and $i\in I$ we obtain from $(\lambda,\alpha_i)=(s_i\lambda,s_i\alpha_i)$ that
\begin{equation}
\label{eq:root}
\langle \lambda,h_i\rangle=\frac{2(\lambda,\alpha_i)}{(\alpha_i,\alpha_i)}.
\end{equation}
In particular we have
\[
(\alpha_i,\alpha_j)=\langle\alpha_j,h_i\rangle\frac{(\alpha_i,\alpha_i)}2\in\BZ,
\]
and hence $(Q,Q)\subset\BZ$.
For $i\in I$ set $t_i=\frac{(\alpha_i,\alpha_i)}2h_i$, and for  $\gamma=\sum_in_i\alpha_i\in Q$ set $t_\gamma=\sum_in_it_i$.
By \eqref{eq:root} we have
$(\lambda,\gamma)=\langle\lambda,t_\gamma\rangle$ for $\lambda\in E$, $\gamma\in Q$.
We fix a $\BZ$-form $\Gh_\BZ$ of $\Gh$
such that
\begin{equation}
\langle\alpha_i,\Gh_\BZ\rangle\subset\BZ,
\qquad
t_i\in\Gh_\BZ\qquad(i\in I).
\end{equation}
We set
\[
P=\{\lambda\in\Gh^*\mid\langle\lambda,\Gh_\BZ\rangle\subset\BZ\},
\qquad
P^+=\{\lambda\in P\mid\langle\lambda,h_i\rangle\in\BZ_{\geqq0}\}.
\]
We fix $\rho\in\Gh^*$ such that $\langle\rho, h_i\rangle=1$ for any $i\in I$, and define a twisted action of $W$ on $\Gh^*$ by
\[
w\circ\lambda=w(\lambda+\rho)-\rho
\qquad(w\in W, \lambda\in \Gh^*).
\]
This action does not depend on the choice of $\rho$, and we have $w\circ P=P$ for any $w\in W$.

Denote by $\CE$ the set of  formal sums $\sum_{\lambda\in P}c_\lambda e(\lambda)$\; ($c_\lambda\in \BZ$) such that 
there exist finitely many $\lambda_1,\dots, \lambda_r\in P$ such that
\[
\{\lambda\in P\mid c_\lambda\ne0\}
\subset \bigcup_{k=1}^r(\lambda_k-Q^+).
\]
Note that $\CE$ is naturally a commutative ring by the multiplication $e(\lambda)e(\mu)=e(\lambda+\mu)$.

Denote by $\Delta^+$ the set of positive roots for the Kac-Moody Lie algebra $\Gg$ associated to the generalized Cartan matrix $(\langle\alpha_j, h_i\rangle)_{i,j\in I}$.
For $\alpha\in\Delta^+$ let $m_\alpha$ be the dimension of the root space of $\Gg$ with weight $\alpha$.
We define an invertible element $D$ of $\CE$ by
\[
D=
\prod_{\alpha\in\Delta^+}(1-e(-\alpha))^{m_\alpha}.
\]

For $n\in\BZ_{\geqq0}$ set 
\[
[n]_x=\frac{x^n-x^{-n}}{x-x^{-1}}\in\BZ[x,x^{-1}],
\quad
[n]!_x=[n]_x[n-1]_x\cdots[1]_x\in\BZ[x,x^{-1}].
\]
We denote by $\BF=\BQ(q)$ the field of rational functions in the variable $q$ with coefficients in $\BQ$.

The quantized enveloping algebra $U$ associated to $\Gh$, $\{h_i\}_{i\in I}$, $\{\alpha_i\}_{i\in I}$, $\Gh_\BZ$, $(\;,\;)$ is the associative algebra over $\BF$ generated by the elements $k_h$, $e_i$, $f_i$ ($h\in\Gh_\BZ$, $i\in I$) satisfying the relations
\begin{align}
&
k_0=1,\qquad
k_hk_{h'}=k_{h+h'}
&(h, h'\in\Gh_\BZ),
\\
&
k_he_ik_{-h}=q_i^{\langle\alpha_i,h\rangle}e_i
&(h\in\Gh_\BZ, i\in I),
\\
&
k_hf_ik_{-h}=q_i^{-\langle\alpha_i,h\rangle}f_i
&(h\in\Gh_\BZ, i\in I),
\\
&
e_if_j-f_je_i=
\delta_{ij}
\frac{k_i-k_i^{-1}}{q_i-q_i^{-1}}
&(i, j\in I),
\\
&
\sum_{r+s=1-\langle\alpha_j,h_i\rangle}
(-1)^re_i^{(r)}e_je_i^{(s)}=0
&(i, j\in I,\;i\ne j),
\\
&
\sum_{r+s=1-\langle\alpha_j,h_i\rangle}
(-1)^rf_i^{(r)}f_jf_i^{(s)}=0
&(i, j\in I,\;i\ne j),
\end{align}
where $k_i=k_{t_i}$, $q_i=q^{(\alpha_i,\alpha_i)/2}$ for $i\in I$, and $e_i^{(r)}=\frac1{[r]!_{q_i}}e_i^r$, $f_i^{(r)}=\frac1{[r]!_{q_i}}f_i^r$ for $i\in I$, $r\in\BZ_{\geqq0}$.
For $\gamma\in Q$ we set $k_\gamma=k_{t_\gamma}$.

We have a Hopf algebra structure of  $U$ given by 
\begin{align}
&
\Delta(k_h)=k_h\otimes k_h,
\\
\nonumber
&
\Delta(e_i)=e_i\otimes1+k_i\otimes e_i,\quad
\Delta(f_i)=f_i\otimes k_i^{-1}+1\otimes f_i
\\
&\varepsilon(k_h)=1,\quad
\varepsilon(e_i)=\varepsilon(f_i)=0,
\\
&
S(k_h)=k_h^{-1},\quad
S(e_i)=-k_i^{-1}e_i,\quad
S(f_i)=-f_ik_i
\end{align}
for $h\in\Gh_\BZ, i\in I$.
We will sometimes use Sweedler's notation for the coproduct;
\[
\Delta(u)=\sum_{(u)}u_{(0)}\otimes u_{(1)}
\qquad(u\in U),
\]
and the iterated coproduct;
\[
\Delta_m(u)=\sum_{(u)_m}u_{(0)}\otimes\cdots\otimes u_{(m)}
\qquad(u\in U).
\]

We define $\BF$-subalgebras $U^0$, $U^+$, $U^-$, $U^{\geqq0}$, $U^{\leqq0}$ of $U$ by
\begin{gather*}
U^0=\langle k_h\mid h\in\Gh_\BZ\rangle,\quad
U^+=\langle e_i\mid i\in I\rangle,\quad
U^-=\langle f_i\mid i\in I\rangle,
\\
U^{\geqq0}=
\langle k_h, e_i\mid h\in\Gh_\BZ, i\in I\rangle,\quad
U^{\leqq0}=
\langle k_h, f_i\mid h\in\Gh_\BZ, i\in I\rangle.
\end{gather*}
For $\gamma\in Q$ set 
\[
U_\gamma=
\{u\in U\mid k_huk_h^{-1}=q^{\langle\gamma,h\rangle}u\;(h\in\Gh_\BZ)\},
\qquad
U^\pm_\gamma=U_\gamma\cap U^\pm.
\]
Then we have
\[
U^0=\bigoplus_{h\in\Gh_\BZ}\BF k_h,
\qquad
U^\pm=\bigoplus_{\gamma\in Q^+}U^{\pm}_{\pm\gamma}.
\]
It is known that the multiplication of $U$ induces isomorphisms
\begin{gather*}
U\cong U^+\otimes U^0\otimes U^-
\cong U^-\otimes U^0\otimes U^+,
\\
U^{\geqq0}\cong U^+\otimes U^0\cong U^0\otimes U^+,
\quad
U^{\leqq0}\cong U^-\otimes U^0\cong U^0\otimes U^-
\end{gather*}
of vector spaces.
It is also known that
\begin{equation}
\label{eq:dim1}
\sum_{\gamma\in Q^+}\dim U^-_{-\gamma}e(-\gamma)
=D^{-1}.
\end{equation}

For a $U$-module $V$ and $\lambda\in P$
we set
\[
V_\lambda=
\{v\in V\mid k_hv=q^{\langle\lambda,h\rangle}v\;(h\in\Gh_\BZ)\}.
\]
We say that a $U$-module $V$ is integrable if $V=\bigoplus_{\lambda\in P}V_\lambda$ and for any $v\in V$ and $i\in I$ there exists some $N>0$ such that $e_i^{(n)}v=f_i^{(n)}v=0$ for $n\geqq N$.

For $i\in I$ and an integrable $U$-module $V$ define an operator $T_i:V\to V$ by
\[
T_iv=
\sum_{-a+b-c=\langle\lambda,h_i\rangle}
(-1)^bq_i^{-ac+b}e_i^{(a)}f_i^{(b)}e_i^{(c)}v\qquad(v\in V_\lambda).
\]
It is invertible, and satisfies $T_iV_\lambda= V_{s_i\lambda}$ for $\lambda\in P$.
There exists a unique algebra automorphism $T_i:U\to U$ such that 
for any integrable $U$-module $V$ we have
$
T_iuv=T_i(u)T_iv$\;($u\in U, v\in V$).
Then we have $T_i(U_\gamma)=U_{s_i\gamma}$ for $\gamma\in Q$.
The action of $T_i$ on $U$ is given by
\begin{gather*}
T_i(k_h)=k_{s_ih},
\quad
T_i(e_i)=-f_ik_i,
\quad
T_i(f_i)=-k_i^{-1}e_i\qquad(h\in\Gh_\BZ),
\\
T_i(e_j)=
\sum_{r+s=-\langle \alpha_j,h_i\rangle}
(-1)^rq_i^{-r}e_i^{(s)}e_je_i^{(r)}\qquad(j\in I, i\ne j),
\\
T_i(f_j)=
\sum_{r+s=-\langle \alpha_j,h_i\rangle}
(-1)^rq_i^{r}f_i^{(r)}f_jf_i^{(s)}\qquad(j\in I, i\ne j)
\end{gather*}
(see \cite[Section 37.1]{Lbook}).

The multiplication of $U$ induces
\begin{align}
\label{eq:td1}
&
U^+\cong (U^+\cap T_i(U^+))\otimes\BF[e_i]
\cong \BF[e_i]\otimes(U^+\cap T_i^{-1}(U^+)),
\\
\label{eq:td2}
&
U^-\cong (U^-\cap T_i(U^-))\otimes\BF[f_i]
\cong
\BF[f_i]\otimes (U^-\cap T_i^{-1}(U^-))
\end{align}
(see \cite[Lemma 38.1.2]{Lbook}).
Moreover, 
\begin{align}
\label{eq:cp1}
&
\Delta(U^+\cap T_i(U^+))\subset U^{\geqq0}\otimes (U^+\cap T_i(U^+)),
\\
\label{eq:cp2}
&
\Delta(U^+\cap T_i^{-1}(U^+))\subset U^0(U^+\cap T_i^{-1}(U^+))
\otimes U^{+},
\\
\label{eq:cp3}
&
\Delta(U^-\cap T_i(U^-))\subset (U^-\cap T_i(U^-))\otimes U^{\leqq0},
\\
\label{eq:cp4}
&
\Delta(U^-\cap T_i^{-1}(U^-))\subset U^{-}\otimes U^0(U^-\cap T_i^{-1}(U^-))
\end{align}
(see \cite[Lemma 2.8]{T1}).

Set
\[
{}^\sharp U^{0}=\bigoplus_{\gamma\in Q}\BF k_\gamma\subset U^0,
\qquad
{}^\sharp U^{\geqq0}={}^\sharp U^{0}U^+,\qquad
{}^\sharp U^{\leqq0}={}^\sharp U^{0}U^-.
\]
They are Hopf subalgebras of $U$.
The Drinfeld pairing is the bilinear form
\[
\tau:{}^\sharp U^{\geqq0}\times {}^\sharp U^{\leqq0}\to\BF
\]
characterized by the following properties:
\begin{align}
\label{eq:D1}
&
\tau(x,y_1y_2)=(\tau\otimes\tau)(\Delta(x),y_1\otimes y_2)
\;\;(x\in {}^\sharp U^{\geqq0}, y_1, y_2\in {}^\sharp U^{\leqq0}),
\\
\label{eq:D2}
&
\tau(x_1x_2,y)=(\tau\otimes\tau)(x_2\otimes x_1,\Delta(y))
\;\;
(x_1, x_2\in {}^\sharp U^{\geqq0}, y\in {}^\sharp U^{\leqq0}),
\\
\label{eq:D3}
&
\tau(k_\gamma,k_\delta)=q^{-(\gamma,\delta)}
\hspace{6cm}(\gamma, \delta\in Q),
\\
\label{eq:D4}
&
\tau(e_i,f_j)=-\delta_{ij}(q_i-q_i^{-1})^{-1}
\hspace{4.5cm}(i, j\in I),
\\
\label{eq:D5}
&
\tau(e_i,k_\gamma)=\tau(k_\gamma,f_i)=0
\hspace{4.3cm}(i\in I, \gamma\in Q).
\end{align}
It satisfies the following properties:
\begin{align}
\label{eq:D6}
&
\tau(xk_\gamma,yk_\delta)=\tau(x,y)q^{-(\gamma,\delta)}
&(x\in U^+, y\in U^-,\gamma, \delta\in Q),
\\
\label{eq:D7}
&
\tau(U^+_\gamma, U^-_{-\delta})=\{0\}
&
(\gamma, \delta\in Q^+, \gamma\ne\delta),
\\
\label{eq:D8}
&
\tau|_{U^+_\gamma\times U^-_{-\gamma}}\;\text{is non-degenerate}
&(\gamma\in Q^+),
\\
\label{eq:D9}
&
\tau(Sx,Sy)=\tau(x,y)
&(x\in {}^\sharp{U}^{\geqq0}, y\in{}^\sharp{U}^{\leqq0}).
\end{align}
Moreover, for $x\in{}^\sharp{U}^{\geqq0}$, $y\in {}^\sharp{U}^{\leqq0}$ we have
\begin{align}
\label{eq:D10}
&xy=\sum_{(x)_2, (y)_2}
\tau(x_{(0)},y_{(0)})\tau(x_{(2)},Sy_{(2)})y_{(1)}x_{(1)},
\\
\label{eq:D11}
&yx=\sum_{(x)_2, (y)_2}
\tau(Sx_{(0)},y_{(0)})\tau(x_{(2)},y_{(2)})x_{(1)}y_{(1)}.
\end{align}
(see \cite[Lemma 2.1.2]{T0}).

For $i\in I$ we define linear maps
\[
r_{i,\pm}:U^\pm\to U^\pm,\qquad
r'_{i,\pm}:U^\pm\to U^\pm
\]
by
\begin{align*}
\Delta(x)\in& r_{i,+}(x)k_i\otimes e_i+\sum_{\delta\in Q^+\setminus\{\alpha_i\}}U^{\geqq0}\otimes U^+_\delta
\qquad(x\in U^+),
\\
\Delta(x)\in& 
e_ik_{\gamma-\alpha_i}\otimes
r'_{i,+}(x)+\sum_{\delta\in Q^+\setminus\{\alpha_i\}}
U^+_\delta U^0\otimes U^{+}
\qquad(x\in U^+_\gamma),
\\
\Delta(y)\in& r_{i,-}(y)\otimes f_ik_{-\gamma+\alpha_i}+\sum_{\delta\in Q^+\setminus\{\alpha_i\}}U^{-}\otimes U^-_{-\delta}U^0
\qquad(y\in U^-_{-\gamma}),
\\
\Delta(y)\in& 
f_i\otimes
r'_{i,-}(y)k_i^{-1}+\sum_{\delta\in Q^+\setminus\{\alpha_i\}}
U^-_{-\delta}\otimes U^{\leqq0}
\qquad(y\in U^{-}).
\end{align*}

We have
\begin{align}
\label{eq:UT1}
U^+\cap T_i(U^+)=&
\{u\in U^+\mid \tau(u,U^-f_i)=\{0\}\}
\\
\nonumber
=&\{u\in U^+\mid r_{i,+}(u)=0\},
\\
\label{eq:UT2}
U^+\cap T_i^{-1}(U^+)=&
\{u\in U^+\mid \tau(u,f_iU^-)=\{0\}\}
\\
\nonumber
=&\{u\in U^+\mid r'_{i,+}(u)=0\},
\\
\label{eq:UT3}
U^-\cap T_i(U^-)=&
\{u\in U^-\mid \tau(U^+e_i,u)=\{0\}\}
\\
\nonumber
=&\{u\in U^-\mid r'_{i,-}(u)=0\},
\\
\label{eq:UT4}
U^-\cap T_i^{-1}(U^-)=&
\{u\in U^-\mid \tau(e_iU^+,u)=\{0\}\}
\\
\nonumber
=&\{u\in U^-\mid r_{i,-}(u)=0\}
\end{align}
(see \cite[Proposition 38.1.6]{Lbook}).

By \eqref{eq:cp1}, \eqref{eq:cp2}, \eqref{eq:cp3}, \eqref{eq:cp4}, \eqref{eq:UT1}, \eqref{eq:UT2}, \eqref{eq:UT3}, \eqref{eq:UT4}
we easily obtain
\begin{align}
\label{eq:D12}
\tau(xe_i^m,yf_i^n)
=&\delta_{mn}
\tau(x,y)
\frac{q_i^{n(n-1)/2}}{(q_i^{-1}-q_i)^n}[n]!_{q_i}
\\
\nonumber
&(x\in U^+\cap T_i(U^+), y\in U^-\cap T_i(U^-)),
\\
\label{eq:D13}
\tau(e_i^mx',f_i^ny')
=&\delta_{mn}
\tau(x',y')
\frac{q_i^{n(n-1)/2}}{(q_i^{-1}-q_i)^n}[n]!_{q_i}
\\
\nonumber
&(x'\in U^+\cap T_i^{-1}(U^+), y'\in U^-\cap T_i^{-1}(U^-)).
\end{align}

We have also 
\begin{align}
\label{eq:D14}
\tau(x,y)=&\tau(T_i^{-1}(x),T_i^{-1}(y))
\\
\nonumber
&
\quad(x\in U^+\cap T_i(U^+), y\in U^-\cap T_i(U^-))
\end{align}
(see \cite[Proposition 38.2.1]{Lbook}, \cite[Theorem 5.1]{T1}).

\section{specialization}
Let $R$ be a subring of $\BF=\BQ(q)$ containing $\BA=\BZ[q,q^{-1}]$.
We denote by $U_R$ the $R$-subalgebra of $U$ generated by $k_h$, $e_i^{(n)}$, $f_i^{(n)}$\;($h\in\Gh_\BZ, i\in I, n\geqq0$).
It is a Hopf algebra over $R$.

We define subalgebras $U_R^0$, $U_R^+$, $U_R^-$, $U_R^{\geqq0}$, $U_R^{\leqq0}$ of $U_R$ by
\begin{gather*}
U_R^0=U^0\cap U_R,\quad
U^{\pm}_R=U^{\pm}\cap U_R,\quad
\\
U_R^{\geqq0}=U^{\geqq0}\cap U_R,\quad
U_R^{\leqq0}=U^{\leqq0}\cap U_R.
\end{gather*}
Setting
$
U_{R,\pm\gamma}^\pm=U_{\pm\gamma}^\pm\cap U_R$ for $\gamma\in Q^+
$
we have
\[
U_{R}^\pm
=\bigoplus_{\gamma\in Q^+}
U_{R,\pm\gamma}^\pm.
\]
It is known that $U^\pm_{R,\pm\gamma}$ is a free $R$-module of rank $\dim U^\pm_{\pm\gamma}$ (see \cite[Section 14.2]{Lbook}).
Hence we have 
\begin{equation}
\label{eq:dim2}
\sum_{\gamma\in Q^+}\rank_R(U^-_{R,-\gamma})e(-\gamma)
=D^{-1}
\end{equation}
by \eqref{eq:dim1}.

The multiplication of $U_R$ induces isomorphisms
\begin{gather*}
U_R\cong U_R^+\otimes U_R^0\otimes U_R^-
\cong U_R^-\otimes U_R^0\otimes U_R^+,
\\
U_R^{\geqq0}\cong U_R^+\otimes U_R^0\cong U_R^0\otimes U_R^+,
\quad
U_R^{\leqq0}\cong U_R^-\otimes U_R^0\cong U_R^0\otimes U_R^-
\end{gather*}
of  $R$-modules.

For $i\in I$ the algebra automorphisms $T_i^{\pm1}:U\to U$ 
preserve $U_R$.
\begin{lemma}
The multiplication of $U_R$ induces isomorphisms
\label{lem:tdR}
\begin{align}
\label{eq:tdR1}
&
U_R^+\cong (U_R^+\cap T_i(U_R^+))\otimes_R
\left(\bigoplus_{n=0}^\infty Re_i^{(n)}\right),
\\
\label{eq:tdR2}
&
U_R^+\cong
\left(\bigoplus_{n=0}^\infty Re_i^{(n)}\right)
\otimes_R
 (U_R^+\cap T_i^{-1}(U_R^+)),
\\
&\label{eq:tdR3}
U_R^-\cong (U_R^-\cap T_i(U_R^-))\otimes_R
\left(\bigoplus_{n=0}^\infty Rf_i^{(n)}\right),
\\
&\label{eq:tdR4}
U_R^-\cong 
\left(\bigoplus_{n=0}^\infty Rf_i^{(n)}\right)
\otimes_R(U_R^-\cap T_i^{-1}(U_R^-)).
\end{align}
\end{lemma}
\begin{proof}
We only show \eqref{eq:tdR1}.
The injectivity of the canonical homomorphism
\[
 (U_R^+\cap T_i(U_R^+))\otimes_R
\left(\bigoplus_{n=0}^\infty Re_i^{(n)}\right)
\to
U_R^+
\]
is clear.
To show the surjectivity it is sufficient to verify that its image is stable under the left multiplication by $e_j^{(n)}$ for any $j\in I$ and $n\geqq0$.
If $j\ne i$, this is clear since $e_j^{(n)}\in  U_R^+\cap T_i(U_R^+)$.
Consider the case $j=i$.
By \eqref{eq:UT1} and the general formula
\[
r_{i,+}(xx')=q_i^{\langle\gamma',\alpha_i^\vee\rangle}r_{i,+}(x)x'+xr_{i,+}(x')
\qquad(x\in U^+, x'\in U^+_{\gamma'})
\]
we easily obtain 
\[
x\in U^+_\gamma\cap T_i(U^+)
\;
\Longrightarrow\;
e_ix-q_i^{\langle\gamma,\alpha_i^\vee\rangle}xe_i\in U_{\gamma+\alpha_i}^+\cap T_i(U^+).
\]
Now let $x\in U^+_{R,\gamma}\cap T_i(U^+_{R})$.
Define $x_k\in U^+_{\gamma+k\alpha_i}\cap T_i(U^+)$ inductively 
by
$x_0=x$, $x_{k+1}=\frac1{[k+1]_{q_i}}(e_ix_k-q_i^{\langle\gamma,\alpha_i^\vee\rangle+2k}x_ke_i)$.
Then we see by induction on $n$ that
\begin{equation}
\label{eq:p1}
e_i^{(n)}x=\sum_{k=0}^n
q_i^{(n-k)(\langle\gamma,\alpha_i^\vee\rangle+k)}
x_ke_i^{(n-k)},
\end{equation}
or equivalently,
\begin{equation}
\label{eq:p2}
x_n=
e_i^{(n)}x-\sum_{k=0}^{n-1}
q_i^{(n-k)(\langle\gamma,\alpha_i^\vee\rangle+k)}
x_ke_i^{(n-k)}.
\end{equation}
We obtain from \eqref{eq:p2} that $x_n\in U^+_R$ by induction on $n$.
By $T_i(U_R)=U_R$ we have $x_n\in U^+_R\cap T_i(U^+)=U^+_R\cap T_i(U^+_R)$.
It follows that
$e_i^{(n)}
(U^+_R\cap T_i(U^+_R))
\subset
\sum_{k=0}^n(U^+_R\cap T_i(U^+_R))e_i^{(k)}$
by \eqref{eq:p1}.
\end{proof}

We set
\[
{}^\sharp {U}_R^0=\bigoplus_{\gamma\in Q}R k_\gamma\subset U_R^0,
\qquad
{}^\sharp {U}_R^{\geqq0}={}^\sharp {U}_R^0U_R^+,
\qquad
{}^\sharp {U}_R^{\leqq0}={}^\sharp {U}_R^0U_R^-.
\]
Define a subring $\tA$ of $\BF$ by
\begin{align}
\tA=&\BZ[q,q^{-1}, (q-q^{-1})^{-1}, [n]_{q}^{-1}\mid n>0]
\\
\nonumber
=&\BZ[q, q^{-1}, (q^n-1)^{-1}\mid n>0].
\end{align}
Then the Drinfeld pairing induces a bilinear form
\[
\tau_\tA:{}^\sharp {U}_\tA^{\geqq0}\times{}^\sharp {U}_\tA^{\leqq0}\to\tA.
\]
For $\gamma\in Q^+$ we denote its restriction to $U^+_{\tA,\gamma}\times U^-_{\tA,-\gamma}$ by
\[
\tau_{\tA,\gamma}:U^+_{\tA,\gamma}\times U^-_{\tA,-\gamma}\to \tA.
\]

In the rest of this paper we fix a field $K$ and $z\in K^\times$ which is not a root of 1, and consider 
the Hopf algebra
\begin{equation}
U_z=K\otimes_\tA U_\tA,
\end{equation}
where $\tA\to K$ is given by $q\mapsto z$.
We define subalgebras $U_z^0$, $U_z^+$, $U_z^-$, $U_z^{\geqq0}$, $U_z^{\leqq0}$ of $U_z$ by
\begin{gather*}
U_z^0=K\otimes_\tA U_\tA^0,\quad
U_z^\pm=K\otimes_\tA U_\tA^\pm,\quad
\\
U_z^{\geqq0}=K\otimes_\tA U_\tA^{\geqq0},\quad
U_z^{\leqq0}=K\otimes_\tA U_\tA^{\leqq0}.
\end{gather*}
For $\gamma\in Q^+$ we set
$U_{z,\pm\gamma}^{\pm}
=K\otimes_\tA U_{\tA,\pm\gamma}^{\pm}$.
Then 
we have
\[
U^0_z=\bigoplus_{h\in\Gh_\BZ}K k_h,\qquad
U_{z}^\pm
=\bigoplus_{\gamma\in Q^+}
U_{z,\pm\gamma}^\pm.
\]
By \eqref{eq:dim2}
we have
\begin{equation}
\label{eq:dim3}
\sum_{\gamma\in Q^+}\dim U^-_{z,-\gamma}e(-\gamma)
=D^{-1}.
\end{equation}

Moreover, setting
\[
U_{z,\gamma}=
\{u\in U_z\mid k_huk_h^{-1}=z^{\langle\gamma,h\rangle}u
\quad(h\in\Gh_\BZ)\}
\qquad(\gamma\in Q),
\]
we have $U_{z,\pm\gamma}^{\pm}=U_z^\pm\cap U_{z,\gamma}$ since $z$ is not a root of $1$.
The multiplication of $U_z$ induces isomorphisms
\begin{gather}
\label{eq:triang}
U_z\cong U_z^+\otimes U_z^0\otimes U_z^-
\cong U_z^-\otimes U_z^0\otimes U_z^+,
\\
U_z^{\geqq0}\cong U_z^+\otimes U_z^0\cong U_z^0\otimes U_z^+,
\quad
U_z^{\leqq0}\cong U_z^-\otimes U_z^0\cong U_z^0\otimes U_z^-
\end{gather}
of $K$-modules.
Here, $\otimes$ denotes $\otimes_K$.

For a $U_z$-module $V$ and $\lambda\in P$
we set
\[
V_\lambda=
\{v\in V\mid k_hv=z^{\langle\lambda,h\rangle}v\;(h\in\Gh_\BZ)\}.
\]
We say that a $U_z$-module $V$ is integrable if $V=\bigoplus_{\lambda\in P}V_\lambda$ and for any $v\in V$ and $i\in I$ there exists some $N>0$ such that $e_i^{(n)}v=f_i^{(n)}v=0$ for $n\geqq N$.

For $i\in I$ and an integrable $U_z$-module $V$ define an operator $T_i:V\to V$ by
\[
T_iv=
\sum_{-a+b-c=\langle\lambda,h_i\rangle}
(-1)^bz_i^{-ac+b}e_i^{(a)}f_i^{(b)}e_i^{(c)}v\qquad(v\in V_\lambda),
\]
where $z_i=z^{(\alpha_i,\alpha_i)/2}$.
It is invertible, and satisfies $T_iV_\lambda= V_{s_i\lambda}$ for $\lambda\in P$.
We denote by $T_i:U_z\to U_z$ the algebra automorphism of $U_z$ induced from $T_i:U_\tA\to U_\tA$.
Then we have $T_i(U_{z,\gamma})=U_{z,s_i\gamma}$ for $\gamma\in Q$.

\begin{lemma}
The multiplication of $U_z$ induces isomorphisms
\label{lem:tdz}
\begin{align}
\label{eq:tdz1}
&
U_z^+\cong (U_z^+\cap T_i(U_z^+))\otimes
\left(\bigoplus_{n=0}^\infty K e_i^{(n)}\right),
\\
\label{eq:tdz2}
&
U_z^+\cong 
\left(\bigoplus_{n=0}^\infty K e_i^{(n)}\right)
\otimes
(U_z^+\cap T_i^{-1}(U_z^+)),
\\
\label{eq:tdz3}
&
U_z^-\cong (U_z^-\cap T_i(U_z^-))\otimes
\left(\bigoplus_{n=0}^\infty K f_i^{(n)}\right),
\\
\label{eq:tdz4}
&
U_z^-\cong 
\left(\bigoplus_{n=0}^\infty K f_i^{(n)}\right)
\otimes
(U_z^-\cap T_i^{-1}(U_z^-)).
\end{align}
\end{lemma}
\begin{proof}
We only show \eqref{eq:tdz1}.
By Lemma \ref{lem:tdR}
we have
\[
U_z^+\cong
\left(K\otimes_\tA(U_\tA^+\cap T_i(U_\tA^+))
\right)
\otimes
\left(\bigoplus_{n=0}^\infty K e_i^{(n)}\right).
\]
By $U_\tA^+\cap T_i(U_\tA^+)=U_\tA^+\cap T_i(U^+)$
the canonical map
$K\otimes_\tA(U_\tA^+\cap T_i(U_\tA^+))
\to
U_z^+\cap T_i(U_z^+)$
is injective.
Hence we have a sequence of canonical maps
\begin{align*}
U_z^+
\cong&
\left(K\otimes_\tA(U_\tA^+\cap T_i(U_\tA^+))
\right)
\otimes
\left(\bigoplus_{n=0}^\infty K e_i^{(n)}\right)
\\
\hookrightarrow&
 (U_z^+\cap T_i(U_z^+))\otimes
\left(\bigoplus_{n=0}^\infty K e_i^{(n)}\right)\to U_z^+.
\end{align*}
Therefore, it is sufficient to show that 
\[
 (U_z^+\cap T_i(U_z^+))\otimes
\left(\bigoplus_{n=0}^\infty K e_i^{(n)}\right)\to U_z
\]
is injective.
This follows by applying $T_i$ to $U_z^+\otimes U_z^{\leqq0}\cong U_z$.
\end{proof}

We set 
\[
{}^\sharp U^0_z=K\otimes_\tA {}^\sharp U^0_\tA,\quad
{}^\sharp U^{\geqq0}_z=K\otimes_\tA {}^\sharp U^{\geqq0}_\tA,
\quad
{}^\sharp U^{\leqq0}_z=K\otimes_\tA {}^\sharp U^{\leqq0}_\tA.
\]
They are Hopf subalgebras of $U_z$.
The Drinfeld pairing induces a bilinear form
\[
\tau_z:{}^\sharp{U}_z^{\geqq0}\times{}^\sharp{U}_z^{\leqq0}\to K.
\]
For $\gamma\in Q^+$ we denote its restriction to $U^+_{z,\gamma}\times U^-_{z,-\gamma}$ by
\[
\tau_{z,\gamma}:U^+_{z,\gamma}\times U^-_{z,-\gamma}\to K.
\]

\section{The modified algebra}
Set
\begin{align*}
J^{+}_{z}
=&
\{x\in U^{+}_{z}\mid
\tau_{z}(x,U^{-}_{z})=\{0\}\},
\\
J^{-}_{z}
=&
\{y\in U^{-}_{z}\mid
\tau_{z}(U^{+}_{z},y)=\{0\}\}.
\end{align*}
For $\gamma\in Q^+$ we set
\[
J_{z,\pm\gamma}^{\pm}=J_z^\pm\cap U^\pm_{z,\pm\gamma}.
\]
By  \eqref{eq:D7}
we have
\begin{equation}
\label{eq:J2}
J_z^\pm=\bigoplus_{\gamma\in Q^+\setminus\{0\}}
J_{z,\pm\gamma}^\pm.
\end{equation}
Define a two-sided ideal $J_z$ of $U_z$ by
\[
J_z=U_zJ^{+}_zU_z+U_zJ_z^{-}U_z.
\]
\begin{proposition}
\label{prop:J}
\begin{itemize}
\item[(i)] 
We have
\[
\Delta(J_z)\subset U_{z}\otimes J_z
+J_z\otimes U_{z},
\quad
\varepsilon(J_z)=\{0\},
\quad
S(J_z)\subset J_z.
\]
\item[(ii)] Under the isomorphism 
$U_z\cong U_z^+\otimes U_z^0\otimes U_z^-$
$($resp. $U_z\cong U_z^-\otimes U_z^0\otimes U_z^+$$)$ 
induced by the multiplication of $U_z$
we have
\[
J_z\cong
J_z^+\otimes U_z^0\otimes U_z^-
+U_z^+\otimes U_z^0\otimes J_z^-,
\]
\[
(\text{resp.}\;
J_z\cong
J_z^-\otimes U_z^0\otimes U_z^+
+U_z^-\otimes U_z^0\otimes J_z^+).
\]
\end{itemize}

\end{proposition}
\begin{proof}
(i) It is sufficient to show
\begin{align}
\label{eq:pJ1}
\Delta(J_z^+)\subset& 
J_z^+ {}^\sharp U_z^0\otimes U^+_z
+
{}^\sharp U_z^{\geqq0}\otimes J_z^+,
\\
\label{eq:pJ2}
\Delta(J_z^-)\subset& 
J_z^-\otimes {}^\sharp U^{\leqq0}_z
+
{}^\sharp U_z^{-}\otimes J_z^-{}^\sharp U_z^0,
\\
\label{eq:pJ3}
\varepsilon(J_z^\pm)=&\{0\},
\\
\label{eq:pJ4}
S(J_z^\pm)\subset& J_z^\pm {}^\sharp U_z^0.
\end{align}
By \eqref{eq:D6} we have 
\[
J^{+}_{z}{}^\sharp U_z^0
=
\{x\in {}^\sharp U^{\geqq0}_{z}\mid
\tau_{z}(x,U^{-}_{z})=\{0\}\}.
\]
Hence in order to verify \eqref{eq:pJ1} it is sufficient to show
\[
\tau_z(\Delta(J_z^+),U_z^-\otimes U_z^-)=\{0\}.
\]
This follows from \eqref{eq:D1}.
The proof of \eqref{eq:pJ2} is similar.
The assertions \eqref{eq:pJ3} and \eqref{eq:pJ4} follow from \eqref{eq:J2} and \eqref{eq:D9}, respectively.

(ii) It is sufficient to show 
\begin{align}
\label{eq:pJ5}
&J_z^\pm U_z^\pm=U_z^\pm J_z^\pm=J_z^\pm,
\\
\label{eq:pJ6}
&J_z^{+}U_z^{\leqq0}=U_z^{\leqq0} J_z^{+},
\qquad
J_z^{-}U_z^{\geqq0}=U_z^{\geqq0} J_z^{-}.
\end{align}
The assertion \eqref{eq:pJ5} follows from 
\eqref{eq:D1}, \eqref{eq:D2}, \eqref{eq:D6}.
By \eqref{eq:J2} we have $J_z^\pm U_z^0=U_z^0J_z^\pm$.
Hence in order to show \eqref{eq:pJ6} it is sufficient to show
$J_z^{+}{}^\sharp U_z^{\leqq0}={}^\sharp U_z^{\leqq0} J_z^{+}$ and 
$
J_z^{-}{}^\sharp U_z^{\geqq0}={}^\sharp U_z^{\geqq0} J_z^{-}$.
Let $x\in J_z^{+}$, $y\in {}^\sharp U_z^{\leqq0}$.
By \eqref{eq:pJ1} we have
\begin{align*}
&\Delta_2(x)
\\
\in&
{}^\sharp U^{\geqq0}_{z}\otimes {}^\sharp U^{\geqq0}_{z}\otimes J_z^{+}
+
{}^\sharp U^{\geqq0}_{z}\otimes J_z^{+}{}^\sharp U_z^0\otimes U_z^++
J_z^{+}{}^\sharp U_z^0\otimes {}^\sharp U^{\geqq0}_{z}\otimes U_z^+.
\end{align*}
Hence we have $xy\in {}^\sharp U_z^{\leqq0}J_z^{+}$ and 
$yx\in J_z^{+}{}^\sharp U_z^{\leqq0}$ by \eqref{eq:D10}, \eqref{eq:D11}.
It follows that $J_z^{+}{}^\sharp U_z^{\leqq0}={}^\sharp U_z^{\leqq0} J_z^{+}$.
The proof of $J_z^{-}{}^\sharp U_z^{\geqq0}={}^\sharp U_z^{\geqq0} J_z^{-}$ is similar.
\end{proof}
By \eqref{eq:D12}, \eqref{eq:D13}, \eqref{eq:D14} we see easily the following.
\begin{lemma}
\label{lem:tdJ}
For $i\in I$ we have
\begin{align*}
J_z^-\cong& (J_z^-\cap T_i(U_z^-))\otimes
\left(\bigoplus_{n=0}^\infty  Kf_i^{(n)}\right),
\\
J_z^-\cong&
\left(\bigoplus_{n=0}^\infty K f_i^{(n)}\right)
\otimes (J_z^-\cap T_i^{-1}(U_z^-)).
\end{align*}
Moreover, we have
\[
T_i^{-1}(J_z^-\cap T_i(U_z^-))=J_z^-\cap T_i^{-1}(U_z^-).
\]
\end{lemma}

We set
\begin{equation}
\overline{U}_z=
U_z/J_z.
\end{equation}
It is a Hopf algebra by Proposition \ref{prop:J}.
Denote by $\overline{U}_z^0$, $\overline{U}_z^\pm$, $\overline{U}_z^{\geqq0}$, $\overline{U}_z^{\leqq0}$, 
${}^\sharp \overline{U}_z^{0}$, 
${}^\sharp \overline{U}_z^{\geqq0}$, ${}^\sharp \overline{U}_z^{\leqq0}$, 
$\overline{U}_{z,\pm\gamma}^\pm$\;($\gamma\in Q^+$) the images of
${U}_z^0$, ${U}_z^\pm$, ${U}_z^{\geqq0}$, ${U}_z^{\leqq0}$, 
${}^\sharp U_z^{0}$, 
${}^\sharp {U}_z^{\geqq0}$, ${}^\sharp {U}_z^{\leqq0}$, 
${U}_{z,\pm\gamma}^\pm$ under $U_z\to\overline{U}_z$ respectively.
By the above argument we have
\[
\overline{U}_z
\cong
\overline{U}_z^+\otimes \overline{U}_z^0\otimes \overline{U}_z^-
\cong
\overline{U}_z^-\otimes \overline{U}_z^0\otimes \overline{U}_z^+,
\]
\[
\overline{U}_z^{\geqq0}
\cong
\overline{U}_z^+\otimes \overline{U}_z^0
\cong
\overline{U}_z^0\otimes \overline{U}_z^+,
\qquad
\overline{U}_z^{\leqq0}
\cong
\overline{U}_z^-\otimes \overline{U}_z^0
\cong
\overline{U}_z^0\otimes \overline{U}_z^-,
\]
\[
{}^\sharp\overline{U}_z^{\geqq0}
\cong
\overline{U}_z^+\otimes {}^\sharp\overline{U}_z^0
\cong
{}^\sharp\overline{U}_z^0\otimes \overline{U}_z^+,
\qquad
{}^\sharp\overline{U}_z^{\leqq0}
\cong
\overline{U}_z^-\otimes {}^\sharp\overline{U}_z^0
\cong
{}^\sharp\overline{U}_z^0\otimes \overline{U}_z^-,
\]
\[
\overline{U}_z^0
\cong
U_z^0=\bigoplus_{h\in\Gh_\BZ}K k_h,
\qquad
{}^\sharp
\overline{U}_z^0
\cong
{}^\sharp
U_z^0
=\bigoplus_{\gamma\in Q}K k_\gamma,
\]
and 
\begin{equation}
\label{eq:decc}
\overline{U}_z^\pm=
\bigoplus_{\gamma\in Q^+}
\overline{U}_{z,\pm\gamma}^\pm,\qquad
\overline{U}_{z,\pm\gamma}^\pm
\cong
U_{z,\pm\gamma}^\pm/J_{z,\pm\gamma}^\pm.
\end{equation}
By definition $\tau_z$ induces a bilinear form
\[
\overline{\tau}_z:{}^\sharp\overline{U}_z^{\geqq0}\times{}^\sharp\overline{U}_z^{\leqq0}\to K
\]
such that for any $\gamma\in Q^+$ its restriction 
\[
\overline{\tau}_{z,\gamma}:\overline{U}^+_{z,\gamma}\times \overline{U}^-_{z,-\gamma}\to K
\]
is non-degenerate.

For $\lambda\in P$ and a $\overline{U}_z$-module $V$ we set
\[
V_\lambda=\{v\in V\mid k_hv=z^{\langle\lambda,h\rangle}v\;(h\in\Gh_\BZ)\}.
\]
We define a category $\CO(\oU_z)$ as follows.
Its objects are $\oU_z$-modules $V$ which satisfy
\begin{equation}
\label{eq:wd}
V=\bigoplus_{\lambda\in P}V_\lambda,
\qquad
\dim V_\lambda<\infty\quad(\lambda\in P),
\end{equation}
and such that
there exist finitely many $\lambda_1,\dots, \lambda_r\in P$ such that
\[
\{\lambda\in P\mid V_\lambda\ne\{0\}\}
\subset \bigcup_{k=1}^r(\lambda_k-Q^+).
\]
The morphisms are homomorphisms of $\oU_z$-modules.

We say that a $\oU_z$-module $V$ is integrable if 
$V=\bigoplus_{\lambda\in P}V_\lambda$ and 
for any $v\in V$ there exists $N>0$ such that for $i\in I$ and $n\geqq N$ we have
$e_i^{(n)}v=f_i^{(n)}v=0$.
We denote by $\CO^{\int}(\oU_z)$ the full subcategory of $\CO(\oU_z)$ consisting of integrable $\oU_z$-modules belonging to $\CO(\oU_z)$.

For each coset $C=\mu+Q\in P/Q$ we denote by $\CO_C(\oU_z)$ the full subcategory of $\CO(\oU_z)$ consisting of 
$V\in \CO_C(\oU_z)$ such that $V=\bigoplus_{\lambda\in C}V_\lambda$.
We also set $\CO_C^{\int}(\oU_z)=\CO_C(\oU_z)\cap\CO^\int(\oU_z)$.
Then we have
\begin{equation}
\CO(\oU_z)=\bigoplus_{C\in P/Q}\CO_C(\oU_z),
\qquad
\CO^{\int}(\oU_z)=\bigoplus_{C\in P/Q}\CO_C^{\int}(\oU_z).
\end{equation}

For $\lambda\in P$ we define $M_z(\lambda)\in\CO_{\lambda+Q}(\oU_z)$ by
\[
M_z(\lambda)=
\oU_z/
\left(
\sum_{h\in\Gh_\BZ}\oU_z(k_h-z^{\langle\lambda,h\rangle})
+\sum_{i\in I}\oU_ze_i
\right),
\]
and for $\lambda\in P^+$ we define $V_z(\lambda)\in\CO_{\lambda+Q}^{\int}(\oU_z)$ by
\[
V_z(\lambda)=
\oU_z/
\left(
\sum_{h\in\Gh_\BZ}\oU_z(k_h-z^{\langle\lambda,h\rangle})
+\sum_{i\in I}\oU_ze_i
+\sum_{i\in I}\oU_zf_i^{(\langle\lambda,h_i\rangle+1)}
\right).
\]

Let $\lambda\in P$.
A $\oU_z$-module $V$ is called a highest weight module with highest weight $\lambda$ if  there exists $v\in V_\lambda\setminus\{0\}$ such that 
$V=\oU_z v$ and $xv=\varepsilon(x)v$\;($x\in \oU_z^+$).
Then we have $V\in\CO_{\lambda+Q}(\oU_z)$.
A $\oU_z$-module is a highest weight module with highest weight $\lambda$ if and only if it is a non-zero quotient of $M_z(\lambda)$.
If there exists an integrable highest weight module with highest weight $\lambda$, then we have $\lambda\in P^+$.
For $\lambda\in P^+$ 
a $\oU_z$-module is an integrable highest weight module with highest weight $\lambda$ if and only if it is a non-zero quotient of $V_z(\lambda)$.

For $V\in\CO(\oU_z)$ we define its formal character by
\[
\ch(V)=
\sum_{\lambda\in P}\dim V_\lambda e(\lambda)\in \CE.
\]
We have 
\[
\ch(M_z(\lambda))=e(\lambda)\overline{D}^{-1}\qquad
(\lambda\in P),
\]
where
\[
\overline{D}^{-1}=
\sum_{\gamma\in Q^+}\dim\oU^-_{z,-\gamma}e(-\gamma)\qquad
(\lambda\in P).
\]

For each coset $C=\mu+Q\in P/Q$ we fix a function $f_C:C\to \BZ$ such that
\[
f_C(\lambda)-f_C(\lambda-\alpha_i)
=2\langle\lambda,t_i\rangle
\qquad(\lambda\in C,i\in I).
\]
\begin{remark}
The function $f_C$ is unique up to addition of a constant function.
If we extend $(\;,\;):E\times E\to\BQ$ to a $W$-invariant symmetric bilinear form on $\Gh^*$, then $f_C$ is given by
\[
f_C(\lambda)=(\lambda+\rho,\lambda+\rho)+a
\qquad(\lambda\in C)
\]
for some $a\in \BQ$.

\end{remark}

For $\gamma\in Q^+$ let
$
\overline{C}_\gamma\in
\overline{U}^+_{z,\gamma}\otimes \overline{U}^-_{z,-\gamma}
$
be
the canonical element of the non-degenerate bilinear form $\overline{\tau}_{z,\gamma}$.
Following Drinfeld we set
\[
{\Omega}_\gamma=
(m\circ (S\otimes 1)\circ P)(\overline{C}_\gamma)
\in \overline{U}^-_{z,-\gamma}\overline{U}_z^0\overline{U}^+_{z,\gamma},
\]
where $m:\overline{U}_z\otimes\overline{U}_z\to \overline{U}_z$ and $P:\overline{U}_z\otimes \overline{U}_z\to \overline{U}_z\otimes \overline{U}_z$ are given by
$m(a,b)=ab$, $P(a\otimes b)=b\otimes a$
(see \cite[Section 3.2]{T0}, \cite[Section 6.1]{Lbook}).
Let $C\in P/Q$.
For $V\in\CO_C(\overline{U}_z)$ we define a linear map
\begin{equation}
\Omega:V\to V
\end{equation}
by
\[
\Omega(v)=z^{f_C(\lambda)}\sum_{\gamma\in Q^+}\Omega_\gamma v
\qquad(v\in V_\lambda).
\]
This operator is called the quantum Casimir operator.
As in \cite[Section 3.2]{T0} we have the following.
\begin{proposition}
Let $C\in P/Q$.
For $\lambda\in C$ the operator $\Omega$ acts on $M_z(\lambda)$ as $z^{f_C(\lambda)}\id$.
\end{proposition}
Since $z$ is not a root of 1, we have
\[
z^{f_C(\lambda)}=z^{f_C(\mu)}
\;\Longrightarrow\;
f_C(\lambda)=f_C(\mu).
\]

\section{Main results}
For $w\in W$ and $x=\sum_{\lambda\in P}c_\lambda e(\lambda)\in\CE$ we set
\[
wx=
\sum_{\lambda\in P}c_\lambda e(w\lambda),
\qquad
w\circ x=
\sum_{\lambda\in P}c_\lambda e(w\circ\lambda).
\]
The elements $wx$, $w\circ x$ may not belong to $\CE$; however, we will only consider the case where $wx, w\circ x\in \CE$.

We denote by
$\sgn:W\to\{\pm1\}$ the character given by $\sgn(s_i)=-1$ for $i\in I$.
\begin{proposition}
For any $w\in W$ we have $w\circ \overline{D}=\sgn(w)\overline{D}$.\end{proposition}
\begin{proof}
We may assume that $w=s_i$ for $i\in I$.
Define $D_i, \overline{D}_i\in\CE$ by
\begin{align*}
D&=(1-e(-\alpha_i))D_i,
\qquad
\overline{D}=(1-e(-\alpha_i))\overline{D}_i.
\end{align*}
Then we have $D_i=\prod_{\alpha\in\Delta^+\setminus\{\alpha_i\}}(1-e(-\alpha))^{m_\alpha}$. 
Moreover, by Lemma \ref{lem:tdz}, Lemma \ref{lem:tdJ} and \eqref{eq:decc} we have
\begin{align*}
D_i^{-1}
=&
\sum_{\gamma\in Q^+}\dim (U^-_{z,-\gamma}\cap T_i(U_z^-))e(-\gamma)
\\
=&
\sum_{\gamma\in Q^+}\dim (U^-_{z,-\gamma}\cap T_i^{-1}(U_z^-))e(-\gamma),
\\
\overline{D}_i^{-1}
=&
D_i^{-1}-\sum_{\gamma\in Q^+}\dim (J^-_{z,-\gamma}\cap T_i(U_z^-))e(-\gamma)
\\
=&
D_i^{-1}-\sum_{\gamma\in Q^+}\dim (J^-_{z,-\gamma}\cap T_i^{-1}(U_z^-))e(-\gamma).
\end{align*}
By $s_i\circ\overline{D}=-(1-e(-\alpha_i))s_i\overline{D}_i$ we have only to show $s_i\overline{D}_i=\overline{D}_i$.
By Lemma \ref{lem:tdJ} we have
\begin{align*}
&
s_i(\sum_{\gamma\in Q^+}\dim (J^-_{z,-\gamma}\cap T_i(U_z^-))e(-\gamma))
\\
=&\sum_{\gamma\in Q^+}\dim (J^-_{z}\cap T_i(U_{z,-\gamma}^-))e(-\gamma)
\\
=&\sum_{\gamma\in Q^+}\dim (J^-_{z,-\gamma}\cap T_i^{-1}(U_z^-))e(-\gamma),
\end{align*}
and hence the assertion follows from
$s_iD_i=D_i$.
\end{proof}

\begin{proposition}
\label{prop:ch}
Let $\lambda\in P^+$.
Assume that $V$ is an integrable highest weight $\oU_z$-module with highest weight $\lambda$.
Then we have
\[
\ch(V)=\sum_{w\in W}\sgn(w)\ch(M_z(w\circ\lambda)).
\]
\end{proposition}
\begin{proof}
The proof below is the same as the one for Lie algebras in Kac \cite[Theorem 10.4]{K}.

Set $C=\lambda+Q\in P/Q$.
Similarly to \cite[Proposition 9.8]{K} we have
\begin{equation}
\label{eq:ch1}
\ch(V)=\sum_{\mu\in\lambda-Q^+, f_C(\mu)=f_C(\lambda)}c_\mu\ch(M_z(\mu))
\qquad
(c_\mu\in\BZ, c_\lambda=1).
\end{equation}
Multiplying \eqref{eq:ch1}  by $\overline{D}$ we obtain
\[
\overline{D}\ch(V)=\sum_{\mu\in\lambda-Q^+, f_C(\mu)=f_C(\lambda)}
c_\mu e(\mu).
\]
Using the action of $T_i\;(i\in I)$ on $V$ we see that $w\ch(V)=\ch(V)$ for $w\in W$, and hence
$w\circ(\overline{D}\ch(V))=\sgn(w)\overline{D}\ch(V)$ for any $w\in W$.
It follows that 
\begin{equation}
\label{eq:ch3}
c_\mu=\sgn(w)c_{w\circ\mu}
\qquad(\mu\in\lambda-Q^+, w\in W).
\end{equation}
Assume that $\mu\in\lambda-Q^+$ satisfies 
$c_\mu\ne0$.
By \eqref{eq:ch3} $W\circ\mu\subset\lambda-Q^+$, and hence we can take 
$\mu'\in W\circ\mu$ such that $\Ht(\lambda-\mu')$ is minimal, where $\Ht(\sum_im_i\alpha_i)=\sum_im_i$.
Then we have $\langle\mu',h_i\rangle\geqq0$ for any $i\in I$ by $s_i\circ\mu'=
\mu'-(\langle\mu',h_i\rangle+1)\alpha_i$ and \eqref{eq:ch3}.
Namely, we have $\mu'\in P^+$.
Then by \cite[Lemma 10.3]{K}  we obtain $\mu'=\lambda$.
\end{proof}
\begin{remark}
I. Heckenberger pointed out to me that Proposition \ref{prop:ch} also follows from the existence of the BGG resolution of integrable
highest weight modules of quantized enveloping algebras given in \cite{HK}
\end{remark}

Recall that any integrable highest weight module $V$ with highest weight $\lambda$ is a quotient of $V_z(\lambda)$.
Proposition \ref{prop:ch} tells us that its character $\ch(V)$ only depends on $\lambda$.
It follows that any  integrable highest weight module with highest weight $\lambda$ is isomorphic to $V_z(\lambda)$.

Consider the case $\lambda=0$.
Since $V_z(0)$ is the trivial one-dimensional module, we obtain the identity
\[
1=
\left(
\sum_{w\in W}\sgn(w)e(w\circ0)
\right)
\left(
\sum_{\gamma\in Q^+}\dim\oU^-_{z,-\gamma}e(-\gamma)
\right)
\]
in $\CE$ by Proposition \ref{prop:ch}.
On the other hand by the corresponding result for the Kac-Moody Lie algebra we have
\[
1=
\left(
\sum_{w\in W}\sgn(w)e(w\circ0)
\right)
\left(
\sum_{\gamma\in Q^+}\dim U^-_{z,-\gamma}e(-\gamma)
\right).
\]
It follows that $U^-_{z,-\gamma}\cong\oU^-_{z,-\gamma}$ for any $\gamma\
\in Q^+$.
By $\dim U^-_{z,-\gamma}=\dim U^+_{z,\gamma}$ and the non-degeneracy of $\overline{\tau}_{z,\gamma}$ we also have
$U^+_{z,\gamma}\cong\oU^+_{z,\gamma}$ for any $\gamma\in Q^+$.
We have obtained the following results.

\begin{theorem}
\label{thm}
The Drinfeld pairing 
\[
\tau_{z,\gamma}:U^+_{z,\gamma}\times U^-_{z,-\gamma}\to K
\]
is non-degenerate for any $\gamma\in Q^+$.
\end{theorem}
\begin{theorem}
\label{thm2}
Let $\lambda\in P^+$.
Assume that $V$ is an integrable highest weight $U_z$-module with highest weight $\lambda$.
Then we have
\[
\ch(V)=D^{-1}\sum_{w\in W}\sgn(w)e(w\circ\lambda).
\]
\end{theorem}

By Theorem \ref{thm} we can define the quantum Casimir operator $\Omega$ for $U_z$.
As in \cite[Section 6.2]{Lbook} we have the following.
\begin{theorem}
\label{th:3}
Any object of $\CO^{\int}(U_z)$ is a direct sum of $V_z(\lambda)$'s for $\lambda\in P^+$.
\end{theorem}

By Theorem \ref{thm} we have the following.
\begin{theorem}
\label{th:4}
Let $\gamma\in Q^+$.
Take bases $\{x_r\}$ and $\{y_s\}$ of $U^+_{\BA,\gamma}$ and $U^-_{\BA,-\gamma}$ respectively, and set $f_\gamma=\det(\tau_{\tA,\gamma}(x_r,y_s))_{r,s}$.
Then we have $f_\gamma\in\tA^\times$.
Namely, we have
\[
f_\gamma=\pm q^a f_1^{\pm1}\cdots f_N^{\pm1},
\]
where $a\in\BZ$, and $f_1,\dots, f_N\in\BZ[q]$ are cyclotomic polynomials.
\end{theorem}
\begin{proof}
We can write $f_\gamma=mgh$, where $m\in\BZ_{>0}$, $g\in \BZ[q]$ is a primitive polynomial with $g(0)>0$ whose irreducible factor is not cyclotomic, and $h\in\tA^\times$.
Note that for any field $K$ and $z\in K^\times$ which is not a root of 1, the specialization of $f_\gamma$ with respect to the ring homomorphism $\tA\to K\;(q\mapsto z)$ is non-zero by Theorem \ref{thm}.
Hence we see easily that $m=1$ and $g=1$.
\end{proof}
In the finite case Theorem \ref{th:4} is well-known (see \cite{KR}, \cite{LS}, \cite{Lbook}).
In the affine case this is a consequence of Damiani \cite{Da1}, \cite{Da2}, where $\det(\tau_{\tA,\gamma}(x_r,y_s))_{r,s}$ is determined explicitly by a case-by-case calculation.

\bibliographystyle{unsrt}

\end{document}